\documentclass[a4paper,11pt]{article}

\usepackage{amsmath,amsthm,amsfonts,amsbsy,amscd,amssymb,graphicx,epsfig,color}

\newtheorem{theorem}{Theorem}[section]
\newtheorem{corollary}{Corollary}

\newtheorem{lemma}[theorem]{Lemma}
\newtheorem{proposition}{Proposition}

\theoremstyle{definition}

\newtheorem{remark}{Remark}

\begin{document}

\title{Convergence of simultaneous distributed-boundary parabolic optimal control problems}

\author{Domingo A.
Tarzia \thanks{Depto. Matem\'atica-CONICET, FCE, Univ. Austral, Paraguay 1950, S2000FZF Rosario, Argentina. E-mail:
DTarzia@austral.edu.ar} \, Carolina M. Bollo \thanks{Depto. Matem\'atica, FCEFQyN, Univ. Nac. de R\'io Cuarto, Ruta 36 Km 601, 5800 R\'io Cuarto, Argentina.
E-mail: cbollo@exa.unrc.edu.ar; cgariboldi@exa.unrc.edu.ar}\,\,  Claudia M. Gariboldi $^\dag$}

\maketitle
\begin{abstract}
We consider a heat conduction problem $S$ with mixed
boundary conditions in a n-dimensional domain $\Omega$ with regular
boundary $\Gamma$ and a family of problems $S_{\alpha}$, where the parameter $\alpha>0$ is
the heat transfer coefficient on the portion of the boundary
$\Gamma_{1}$. In relation to these state systems, we formulate
simultaneous \emph{distributed-boundary} optimal control
problems on the internal energy $g$ and the heat flux $q$ on the complementary portion of the boundary $\Gamma_{2}$. We obtain
existence and uniqueness of the optimal controls, the first order
optimality conditions in terms of the adjoint state and the
convergence of the optimal controls, the system and the
adjoint states when the heat transfer coefficient $\alpha$ goes to
infinity. Finally, we prove estimations between the simultaneous distributed-boundary optimal control and the distributed optimal control
problem studied in a previous paper of the first author.
\end{abstract}

\textbf{2010 Mathematic Subject Classification:} Primary: 49J20;
Secondary: 35K05, 49K20.

\textbf{keywords:} Parabolic variational equalities, Optimal
control, Mixed boundary conditions, Existence and uniqueness,
Optimality conditions.

%The title of your section 1
\section{Introduction}
We consider a bounded domain $\Omega $ in ${\Bbb R}^{n}$, whose
regular boundary $\Gamma $ consists of the union of the two disjoint
portions $\Gamma _{1}$ and $\Gamma _{2}$ with $|\Gamma_{1}|>0$ and
$|\Gamma_{2}|>0$. We denote with $|\Gamma_i|=meas(\Gamma_i)$ (for
$i=1,2$), the $(n-1)$-dimensional Hausdorff measure of the portion
$\Gamma_i$ on $\Gamma$. Let $[0,T]$ a time interval, for a $T>0$. We
present the following heat conduction problems $S$ and $S_{\alpha }$
(for each parameter $\alpha >0)$ respectively, with mixed boundary
conditions (we denote by $u(t)$ to the function $u(\cdot,t)$):
\begin{equation}
\frac{\partial u}{\partial t}-\Delta u=g\,\ \text{ in }\Omega \ \
\,\,\,\,\,\,u\big|_{\Gamma _{1}}=b\,\,\,\,\,\,\,\,-\frac{\partial
u}{\partial n}\bigg|_{\Gamma _{2}}=q\,\,\,\,\,\,\,\,u(0)=v_b
\label{P}
\end{equation}
\begin{equation}
\frac{\partial u}{\partial t}-\Delta u=g\,\ \text{ in }\Omega \ \ \,\,\,\,\,\,-\frac{\partial u}{\partial n%
}\bigg|_{\Gamma _{1}}=\alpha (u-b)\,\,\,\,\,\,\,\,-\frac{\partial
u}{\partial n}\bigg|_{\Gamma _{2}}=q \,\,\,\,\,\,\,\,u(0)=v_b
\label{Palfa}
\end{equation}
where $u$ is the temperature in $\Omega\times (0,T)$, $g$ is the
internal energy in $\Omega $, $b$ is the temperature on $
\Gamma_{1}$ for (\ref{P}) and the temperature of the external
neighborhood of $\Gamma_{1}$ for (\ref{Palfa}), $v_{b}=b$ on
$\Gamma_{1}$, $q$ is the heat flux on $\Gamma_{2}$ and $\alpha
>0$ is the heat transfer coefficient on $\Gamma_{1}$,
which satisfy the hypothesis: $g\in
\mathcal{H}=L^2(0,T;L^2(\Omega))$, $q\in
\mathcal{Q}=L^2(0,T;L^2(\Gamma_2))$, $b\in
H^{\frac{1}{2}}(\Gamma_1)$ and $v_{b}\in H^{1}(\Omega)$.

Let $u$ and $u_{{\alpha}}$ the unique solutions of the parabolic
problems (\ref{P}) and (\ref{Palfa}), whose variational formulations
are given by \cite{Li,MT}:
\begin{equation}\label{Pvariacional}
\left\{
\begin{array}{l l}
u-v_b \in L^2(0,T; V_0), \qquad u(0)=v_b\quad\text{and}\quad \dot{u}\in L^2(0,T; V_0')\\
\text{such that}\quad \langle \dot{u}(t), v\rangle+a(u(t),v)=L(t,v),
\quad \forall v\in V_0,
\end{array}
\right.
\end{equation}
\begin{equation}\label{Palfavariacional}
\left\{
\begin{array}{l l}
u_{\alpha} \in L^2(0,T; V), \qquad u_{\alpha}(0)=v_b\quad\text{and}\quad \dot{u}_{\alpha}\in L^2(0,T; V')\\
\text{such that}\quad \langle \dot{u}_{\alpha}(t),
v\rangle+a_{\alpha}(u_{\alpha}(t),v)=L_{\alpha}(t,v), \quad \forall
v\in V,
\end{array}
\right.
\end{equation}
where $\langle\cdot,\cdot\rangle$ denote the duality between the
functional space ($V$ or $V_{0}$) and its dual space ($V'$ or
$V'_{0}$) and
\[
V=H^{1}(\Omega )\,;\,\,\,\,\,\,\,\,\,V_{0}=\{v\in V:\,v\big|_{\Gamma
_{1}}=0\}\,;\,\,\,\,\,\,\,\,\,Q=L^2(\Gamma_2); \quad
H=L^2(\Omega)\,;
\]
\[
(g,h)_{H}=\int_{\Omega}gh dx\,; \quad (q,\eta)_{Q}=\int_{\Gamma
_{2}}q\eta d\gamma\,;
\]
\[
a(u,v)=\int_{\Omega }\nabla u\nabla vdx\,;\quad a_{\alpha
}(u,v)=a(u,v)+\alpha \int\limits_{\Gamma_1}u vd\gamma\,;
\]
\[
L(t,v)=(g(t),v)_{H}-(q(t),v)_{Q};\quad L_{{\alpha}}(t,v)=\,L
(t,v)+\alpha \int\limits_{\Gamma_1}b vd\gamma,
\]
\[
\exists \lambda_{0}>0\quad \text{such that}\quad a(v,v)\geq \lambda_{0}||v||^{2}_{V},\quad \forall v\in V_{0}.
\]

We consider $\mathcal{H}=L^2(0,T;H)$, with norm
$||.||_{\mathcal{H}}$ and internal product
\[
(g,h)_{\mathcal{H}}=\int\limits_0^T(g(t),h(t))_{H} dt,
\]
and the space $\mathcal{Q}=L^2(0,T;Q)$, with norm
$||.||_{\mathcal{Q}}$ and internal product
\[
(q,\eta)_{\mathcal{Q}}=\int\limits_0^T(q(t),\eta(t))_{Q} dt.
\]

For the sake of simplicity, for a Banach space $X$ and $1\leq
p\leq\infty$, we will often use $L^{p}(X)$ instead of
$L^{p}(0,T;X)$.

If we denote by $u_{gq}$ and $u_{\alpha gq}$ the unique solution of
the problems (\ref{Pvariacional}) and (\ref{Palfavariacional})
respectively, we formulate the following simultaneous
\emph{distributed-boundary} optimal control problems $P$ and $P_{\alpha}$ on the internal
energy $g$ and the heat flux $q$, as a vector control variable, respectively \cite{GT2,Li,Tro}:
\begin{equation}\label{pcd1bi}
\text{find}\quad
(\overline{\overline{g}},\overline{\overline{q}})\in
\mathcal{H}\times\mathcal{Q}\quad\text{such that}\quad
J(\overline{\overline{g}},\overline{\overline{q}})=\min\limits_{g\in
\mathcal{H},q\in \mathcal{Q}}J(g,q)
\end{equation}
\begin{equation}\label{pcd2bi}
\text{find}\quad
(\overline{\overline{g}}_{\alpha},\overline{\overline{q}}_{\alpha})\in
\mathcal{H}\times\mathcal{Q}\quad\text{such that}\quad
J_{\alpha}(\overline{\overline{g}}_{\alpha},\overline{\overline{q}}_{\alpha})=\min\limits_{g\in
\mathcal{H},q\in \mathcal{Q}}J_{\alpha}(g,q),
\end{equation}
where the cost functionals $J$ and $J_{\alpha}$ are given by
\begin{equation}\label{Jbi}
J(g,q)=\frac{1}{2}||u_{gq}-z_d||^2_{\mathcal{H}}+\frac{M_1}{2}||g||^2_{\mathcal{H}}+\frac{M_2}{2}||q||^2_{\mathcal{Q}}
\end{equation}
\begin{equation}\label{Jalphabi}J_{\alpha }(g,q)=\frac{1}{2}||u_{\alpha gq}-z_d||^2_{\mathcal{H}}+\frac{M_1}{2}||g||^2_{\mathcal{H}}+\frac{M_2}{2}||q||^2_{\mathcal{Q}},
\end{equation}
with $z_d \in \mathcal{H}$ given and $M_1,\,M_2$ positive
constants.

\par In \cite{GT1}, the authors studied boundary optimal control problems
on the heat flux  $q$ in mixed elliptic problems  and they proved
existence, uniqueness and asymptotic behavior to the optimal
solutions, when the heat transfer coefficient goes to infinity.
Similar results were obtained in \cite{GT2} for simultaneous
distributed-boundary optimal control problems on the internal energy
$g$ and the heat flux $q$ in mixed elliptic problems. In \cite{MT},
convergence results were proved for non-stationary heat conduction
problems in relation to distributed optimal control problems on the
internal energy $g$ as a control variable. In \cite{BT1} and
\cite{BT2}, were studied control problems on the source $g$ and the
flux $q$ respectively, for parabolic variational inequalities of
second kind. Other papers on the subject are
\cite{GS,GLMOP,GMO,SS,SB,SM,SA,V,WY,ZDB}. Our interest is the
convergence when $\alpha\rightarrow \infty$, which is related to
\cite{BFR,TT,T}. Variational inequalities was popular in the $70's$,
most of the main techniques for parabolic variational inequalities
can be found in \cite{BL}. It is well know that the regularity of
the mixed problem is problematic when both portions of the boundary
$\Gamma_{1}$ and $\Gamma_{2}$ have a nonempty intersection, e.g. see
the book \cite{G}. Sufficient conditions (on the data) to obtain a
$H^{2}$ regularity for an elliptic mixed boundary condition are given
in \cite{BBP}, see also \cite{AK}, among others. Numerical analysis
of a parabolic PDE with mixed boundary conditions (Dirichlet and
Neumann) is studied in \cite{BO}, while a parabolic control problem
with Robin boundary conditions is considered in \cite{BT,CGH}. In
the present paper, in Section 2 and Section 3, we study simultaneous
distributed-boundary optimal control for heat conduction problems
(\ref{P}),(\ref{pcd1bi}) and (\ref{Jbi}), and (\ref{Palfa}),
(\ref{pcd2bi}) and (\ref{Jalphabi}), respectively. We obtain
existence and  uniqueness results of the optimal controls and we
give the first order optimality condition in terms of the adjoint
states of the systems. In Section 4, we prove convergence results of
the optimal controls, and the system and adjoint states
corresponding to the problems (\ref{Palfa}), (\ref{pcd2bi}) and
(\ref{Jalphabi}), when the heat transfer coefficient $\alpha$ goes
to infinity. In Section 5, we study the relation between the
solutions of the distributed optimal control problems given in
\cite{MT} and the first component to the simultaneous
distributed-boundary optimal control problems (\ref{pcd1bi}) and
(\ref{pcd2bi}). Finally, we give a characterization of the
simultaneous optimal controls by using fixed point theory.

\section{System $S$ and its Corresponding Distributed-Boundary Optimal Control Problem}

Here, we prove the existence and uniqueness of the simultaneous
distributed-boundary optimal control
 $(\overline{\overline{g}},\overline{\overline{q}})$ for the optimal control problem (\ref{pcd1bi}) and
we give the optimality condition en terms of the adjoint state
$p_{\overline{\overline{g}}\,\overline{\overline{q}}}$.

Following \cite{GT2,Li,MT}, we
 define the application $C:\mathcal{H}\times\mathcal{Q}\rightarrow L^2 (V_0)$
 by $C(g,q)=u_{gq}-u_{00}$ where $u_{00}$ is the solution of the problem
(\ref{Pvariacional}) for $g=0$ and $q=0$. Moreover, we consider
$\Pi:(\mathcal{H}\times\mathcal{Q})\times
(\mathcal{H}\times\mathcal{Q})\rightarrow \mathbb{R}$ and
$\mathcal{L}:\mathcal{H}\times\mathcal{Q}\rightarrow \mathbb{R}$
defined by
\begin{equation*}\Pi((g,q),(h,\eta))=(C(g,q),C(h,\eta))_\mathcal{H}+M_1(g,h)_{\mathcal{H}}+M_2(q,\eta)_\mathcal{Q},\end{equation*}
\begin{equation*}\mathcal{L}(g,q)=(C(g,q),z_d-u_{00})_\mathcal{H}, \qquad\forall (g,q),\,(h,\eta)\in \mathcal{H}\times\mathcal{Q}\end{equation*}
and we give the following result.
\begin{lemma}\label{lema1bi}
\begin{enumerate}
                 \item [i)] $C$ is a linear and continuous application.
                 \item [ii)] $\Pi$ is a bilinear, symmetric, continuous
                 and coercive form.
                 \item [iii)] $\mathcal{L}$ is linear and continuous application in $\mathcal{H}\times\mathcal{Q}$.
                 \item [iv)] $J$ can be write as:
\[
                 J(g,q)=\frac{1}{2}\Pi((g,q),(g,q))-\mathcal{L}(g,q)+\frac{1}{2}||u_{00}-z_d||^2_{\mathcal{H}}, \quad  \forall (g,q)\,\in \mathcal{H}\times\mathcal{Q}.
                 \]
    \item [v)]  $J$ is a coercive functional on
                 $\mathcal{H}\times\mathcal{Q}$, that is, there exists $N>0$ such
                 that $\forall (g_2,q_2),\,(g_1,q_1)\in\mathcal{H}\times
                    \mathcal{Q},\,\,\forall t\in[0,1]$
                 \begin{equation*}\begin{split}
                    &\quad(1-t)J(g_2,q_2)+tJ(g_1,q_1)-J((1-t)(g_2,q_2)+t(g_1,q_1))\\&=
                    \frac{t(1-t)}{2}\left[||u_{g_2q_2}-u_{g_1q_1}||^2_{\mathcal{H}}+M_1||g_2-g_1||^2_{\mathcal{H}}+M_2||q_2-q_1||^2_{\mathcal{Q}}\right]\\&\geq
                    \frac{Nt(1-t)}{2}||(g_2-g_1,q_2-q_1)||^2_{\mathcal{H}\times\mathcal{Q}}.
                 \end{split}\end{equation*}

                 \item [vi)] There exists a unique optimal control $(\overline{\overline{g}},\overline{\overline{q}})\in
                 \mathcal{H}\times\mathcal{Q}$ such that (\ref{pcd1bi}) holds.
\end{enumerate}
\end{lemma}
\begin{proof}
(i) Following \cite{MT}, we obtain that there exist different positive constants $K$ such that
\begin{equation}\nonumber
    ||\nabla C(g,q)||_\mathcal{H}
\leq
K ||(g,q)||_{\mathcal{H}\times \mathcal{Q}}.
\end{equation}
\begin{equation*} ||C (g,q)||_{L^{\infty}(H)}\leq
K||(g,q)||_{\mathcal{H}\times \mathcal{Q}}.
\end{equation*}
\begin{equation*}
    \left[\int_{0}^{T}\bigg|\bigg|\frac{d}{dt}C(g,q)(t)\bigg|\bigg|^{2}_{V_0'}dt\right]^{1/2}
   \leq   K||(g,q)||_{\mathcal{H}\times \mathcal{Q}}.
  \end{equation*}
Therefore, $C:\mathcal{H}\times \mathcal{Q}\rightarrow \{v\in L^2(V_0)\cap L^{\infty}(H): \dot{v}\in
L^2(V_0')\}$ is a continuous operator.

(ii)- (v) It follows from the definition of $J$, $\Pi$ and $\mathcal{L}$ and a similar way that \cite{GT2}.

(vi)  It follows taking into account \cite{GT2,Li}.
\end{proof}

We define the adjoint state $p_{gq}$ corresponding to problem
(\ref{Pvariacional}) for each  $(g,q)\in
\mathcal{H}\times\mathcal{Q}$, as the unique solution of the
variational equality
\begin{equation}\label{ecvarbi}
\left\{
\begin{array}{l l}
p_{gq}\in L^2(V_0), \quad p_{gq}(T)=0\quad\text{ and }\quad
\dot{p}_{gq}\in L^2(V_0') \\ \text{such that} \,\, -\langle
\dot{p}_{gq}(t), v\rangle+a(p_{gq}(t),v)=(u_{gq}(t)-z_d,v)_H, \,\,
\forall v\in V_0.
\end{array}
\right.
\end{equation}
\begin{lemma}\label{lema22bi}
\quad
\begin{enumerate}\item [i)] The adjoint state $p_{gq}$ satisfies
the following equality
\begin{small}$$(C(h,\eta),u_{gq}-z_d)_\mathcal{H}=(h,p_{gq})_\mathcal{H}-(\eta,p_{gq})_\mathcal{Q}.$$\end{small}
                 \item [ii)] $J$ is G\^{a}teaux differentiable
                  and $J^{\prime}$ is given by: $\forall (g,q),(h,\eta)\in \mathcal{H}\times
                  \mathcal{Q}$
                   \begin{equation*}\begin{split}
                  &\quad\langle J^{\prime}(g,q),(h-g,\eta-q)\rangle\\&=(u_{h\eta}-u_{gq},u_{gq}-z_d)_\mathcal{H}
                  +M_1(g,h-g)_\mathcal{H}+M_2(q,\eta-q)_\mathcal{Q}\\&=\Pi((g,q),(h-g,\eta-q))-\mathcal{L}(h-g,\eta-q).
                   \end{split}\end{equation*}
                 \item [iii)] The G\^{a}teaux derivative of $J$ can be write as: $\forall(h,\eta)\in\mathcal{H}\times\mathcal{Q}$
                 \begin{equation*}\begin{split}
                  \langle J^{\prime}(g,q),(h,\eta)\rangle= (M_1g+p_{gq},h)_\mathcal{H}+(M_2q-p_{gq},\eta)_\mathcal{Q}.
                 \end{split}\end{equation*}
                 \item [iv)] The optimality condition for the
                 problem (\ref{pcd1bi}) is:
                 $\forall(h,\eta)\in\mathcal{H}\times\mathcal{Q}$
                 \begin{equation*}\begin{split}
                  \langle J^{\prime}(\overline{\overline{g}},\overline{\overline{q}}),(h,\eta)\rangle=
                  (M_1\overline{\overline{g}}+p_{\overline{\overline{g}}\,\overline{\overline{q}}},h)_\mathcal{H}+
                  (M_2\overline{\overline{q}}&-p_{\overline{\overline{g}}\,\overline{\overline{q}}},\eta)_\mathcal{Q}=0.
                 \end{split}\end{equation*}\end{enumerate}
\end{lemma}
\begin{proof} (i) Following \cite{MT}, if we take $v=C(h,\eta)(t)\in V_0$ in (\ref{ecvarbi}) and we integrate between $0$ and $T$, we have \begin{equation}
 -\left( \dot{p}_{gq},
C(h,\eta)\right)_\mathcal{H}+\int_0^Ta(p_{pq}(t),C(h,\eta)(t))dt=(u_{gq}-z_d,C(h,\eta))_\mathcal{H}.\nonumber
\end{equation}
On the other hand, taking $v=p_{gq}(t)$ in (\ref{Pvariacional}), for
$g=0,\,q=0$ and  $g=h,\,q=\eta$, we obtain \small{\begin{equation*}
( \dot{u}_{h\eta}(t)-\dot{u}_{00}(t),
p_{gq}(t))_H+a(u_{h\eta}(t)-u_{00}(t),p_{gq}(t))=(h(t),p_{gq}(t))_H-(\eta(t),p_{gq}(t))_Q
\end{equation*}}
and integrating between $0$ and $T$, we have
\begin{equation} \begin{split}\left(
\dot{u}_{h\eta}-\dot{u}_{00},
p_{gq}\right)_{\mathcal{H}}&+\int_0^Ta(C(h,\eta)(t),p_{gq}(t))dt=(h,p_{gq})_{\mathcal{H}}-(\eta,p_{gq})_{\mathcal{Q}}.
\end{split}\nonumber\end{equation}
Therefore \small{\begin{equation}\begin{split}
 &(C(h,\eta),u_{gq}-z_d)_\mathcal{H}=
-\int_0^T\frac{d}{dt}\left( p_{gq}(t),
C(h,\eta)(t)\right)_Hdt+(h,p_{gq})_\mathcal{H}-(\eta,p_{gq})_\mathcal{Q}
\end{split}\nonumber\end{equation}}
and by using that $p_{gq}(T)=0$ and $C(h,\eta)(0)=0$, we have (i).

(ii) It follows in similar way that \cite{GT2,MT}.

(iii) From (i) and (ii), we obtain
$$\langle J^{\prime}(g,q),(h,\eta)\rangle = (M_1g+p_{gq},h)_\mathcal{H}+(M_2q-p_{gq},\eta)_\mathcal{Q},\quad
\forall(h,\eta)\in \mathcal{H}\times \mathcal{Q}.$$

(iv)  From (iii), we have
$$(M_1\overline{\overline{g}}+p_{\overline{\overline{g}}\,\overline{\overline{q}}},h)_\mathcal{H}
+(M_2\overline{\overline{q}}-p_{\overline{\overline{g}}\,\overline{\overline{q}}},\eta)_\mathcal{Q}=0,\qquad\forall(h,\eta)\in\mathcal{H}\times\mathcal{Q}.$$
\end{proof}

\section{System $S_{\alpha}$ and its Corresponding Distributed-Boundary Optimal Control Problem}

We will prove, for each $\alpha >0$, the existence and uniqueness of
the simultaneous distributed-boundary optimal control
$(\overline{\overline{g}}_{\alpha},\overline{\overline{q}}_{\alpha})\in
\mathcal{H}\times\mathcal{Q}$ for the problem (\ref{pcd2bi}) and we
will give the optimality condition in terms of the adjoint state
$p_{\alpha\overline{\overline{g}}_{\alpha}\overline{\overline{q}}_{\alpha}}$.
For this purpose, following \cite{GT2,Li,MT}, we define
$C_{\alpha}:\mathcal{H}\times\mathcal{Q}\rightarrow L^2 (V)$ given
by $C_{\alpha}(g,q)=u_{\alpha gq}-u_{\alpha 00}$, where $u_{\alpha
00}$ is the solution of the variational problem
(\ref{Palfavariacional}) for $g=0$ and $q=0$, and
$\Pi_{\alpha}:(\mathcal{H}\times\mathcal{Q})\times
(\mathcal{H}\times\mathcal{Q})\rightarrow \mathbb{R}$ and
$\mathcal{L}_{\alpha}:\mathcal{H}\times\mathcal{Q}\rightarrow
\mathbb{R}$ are defined by
\begin{equation*}\Pi_{\alpha}((g,q),(h,\eta))=(C_{\alpha}(g,q),C_{\alpha}(h,\eta))_\mathcal{H}+M_1(g,h)_{\mathcal{H}}+M_2(q,\eta)_\mathcal{Q},\end{equation*}
\begin{equation*}\mathcal{L}_{\alpha}(g,q)=(C_{\alpha}(g,q),z_d-u_{\alpha 00})_\mathcal{H} \qquad\quad\forall (g,q),\,(h,\eta)\in \mathcal{H}\times\mathcal{Q}.\end{equation*}
\begin{lemma}\label{lema1bia}
\quad
\begin{enumerate}
                 \item [i)] $C_{\alpha}$ is a linear and continuous application.
                 \item [ii)] $\Pi_{\alpha}$ is bilinear, symmetric, continuous and coercive form.
                 \item [iii)] $\mathcal{L}_{\alpha}$ is linear and continuous in $\mathcal{H}\times\mathcal{Q}.$
                 \item [iv)] $J_{\alpha}$ can be write as \small{$$J_{\alpha}(g,q)=\frac{1}{2}\Pi_{\alpha}((g,q),(g,q))-
                 \mathcal{L}_{\alpha}(g,q)+\frac{1}{2}||u_{\alpha 00}-z_d||^2_{\mathcal{H}},\,\forall (g,q)\in
                 \mathcal{H}\times\mathcal{Q}.$$}
                \item [v)] $J_{\alpha}$ is a coercive functional on
                 $\mathcal{H}\times\mathcal{Q}$, that is, there exists $\overline{N}>0$ such
                 that $\forall (g_2,q_2),\,(g_1,q_1)\in\mathcal{H}\times
                    \mathcal{Q},\,\,\forall t\in[0,1]$
                 \begin{equation*}\begin{split}
                    &\quad(1-t)J_{\alpha}(g_2,q_2)+tJ_{\alpha}(g_1,q_1)-J_{\alpha}((1-t)(g_2,q_2)+t(g_1,q_1))\\&=
                    \frac{t(1-t)}{2}\left[||u_{g_2q_2}-u_{g_1q_1}||^2_{\mathcal{H}}+M_1||g_2-g_1||^2_{\mathcal{H}}+M_2||q_2-q_1||^2_{\mathcal{Q}}\right]\\&\geq
                    \frac{\overline{N}t(1-t)}{2}||(g_2-g_1,q_2-q_1)||^2_{\mathcal{H}\times\mathcal{Q}}.
                 \end{split}\end{equation*}
                 \item [vi)] There exists a unique optimal control
$(\overline{\overline{g}}_{\alpha},\overline{\overline{q}}_{\alpha})$
such that (\ref{pcd2bi}) holds.
               \end{enumerate}
\end{lemma}

\begin{proof}
This results in a similar way that Lemma \ref{lema1bi} and the proof is omitted.
\end{proof}

We define the adjoint state $p_{\alpha gq}$ corresponding to
(\ref{Palfavariacional}) for each  $(g,q)\in
\mathcal{H}\times\mathcal{Q}$, as the unique solution of
\small{\begin{equation}\label{palphavariacional} \left\{
\begin{array}{l l}
p_{\alpha gq} \in L^2(V), \quad p_{\alpha
gq}(T)=0\quad\text{and}\quad \dot{p}_{\alpha gq}\in L^2(V') \\
\text{such that}\, -\langle \dot{p}_{\alpha gq}(t),
v\rangle+a_{\alpha}(p_{\alpha gq}(t),v)=(u_{\alpha gq}(t)-z_d,v)_H,
\, \forall v\in V
\end{array}
\right.
\end{equation}}
and for each $\alpha >0$, we obtain analogous properties to Lemma \ref{lema22bi}, whose proof will be also omitted.
\begin{lemma}\label{lema32bi}
\quad
\begin{enumerate}\item [i)] The adjoint state $p_{\alpha gq}$ satisfies
the following equality
\begin{small}$$(C(h,\eta),u_{\alpha gq}-z_d)_\mathcal{H}=(h,p_{\alpha gq})_\mathcal{H}-(\eta,p_{\alpha gq})_\mathcal{Q}.$$\end{small}
                 \item [ii)] $J_{\alpha}$ is G\^{a}teaux differentiable
                  and $J_{\alpha}^{\prime}$ is given by $\forall (g,q),(h,\eta)\in \mathcal{H}\times
                  \mathcal{Q}$
                   \begin{equation*}\begin{split}
                  &\quad\langle J_{\alpha}^{\prime}(g,q),(h-g,\eta-q)\rangle\\&=(u_{\alpha h\eta}-u_{\alpha gq},u_{\alpha gq}-z_d)_\mathcal{H}
                  +M_1(g,h-g)_\mathcal{H}+M_2(q,\eta-q)_\mathcal{Q}\\&=\Pi_{\alpha}((g,q),(h-g,\eta-q))-\mathcal{L}_{\alpha}(h-g,\eta-q).
                   \end{split}\end{equation*}
                 \item [iii)] The G\^{a}teaux derivative of $J_{\alpha}$ can be write
                 as, $\forall(h,\eta)\in\mathcal{H}\times\mathcal{Q}$
                 \begin{equation*}\begin{split}
                  \langle J_{\alpha}^{\prime}(g,q),(h,\eta)\rangle= (M_1g+p_{\alpha gq},h)_\mathcal{H}+(M_2q-p_{\alpha gq},\eta)_\mathcal{Q}.
                 \end{split}\end{equation*}
                 \item [iv)] The optimality condition for the
                 problem (\ref{pcd2bi}) is,
                 $\forall(h,\eta)\in\mathcal{H}\times\mathcal{Q}$
                 \begin{equation*}
                  \langle J_{\alpha}^{\prime}(\overline{\overline{g}}_{\alpha},\overline{\overline{q}}_{\alpha}),(h,\eta)\rangle=
                  (M_1\overline{\overline{g}}_{\alpha}+p_{\alpha\overline{\overline{g}}_{\alpha}\,\overline{\overline{q}}_{\alpha}},h)_\mathcal{H}+
                  (M_2\overline{\overline{q}}_{\alpha}-p_{\alpha\overline{\overline{g}}_{\alpha}\,\overline{\overline{q}}_{\alpha}},\eta)_\mathcal{Q}=0.
                 \end{equation*}\end{enumerate}
\end{lemma}

\section{Convergence of Simultaneous Distributed Boundary Optimal Control Problems when $\alpha\rightarrow\infty$}

For fixed $(g,q)\in \mathcal{H}\times\mathcal{Q}$, we can prove
estimations for $u_{\alpha gq}$ and $p_{\alpha gq}$ and we obtain
the strong convergence to $u_{gq}$ and $p_{gq}$, respectively, when
$\alpha$ goes to infinity.
\begin{proposition}\label{proposicion41bi} For fixed $(g,q)$, if $u_{\alpha gq}$ is the unique solution of the variational equality (\ref{Palfavariacional}),
we have
\begin{small}\begin{equation}\label{estimacion1bi}
||\dot{u}_{\alpha gq}||_{L^{2}(V_0')}+   ||u_{\alpha
gq}||_{L^{\infty}(H)}+
    ||u_{\alpha gq}||_{L^{2}(V)}
     +\sqrt{(\alpha-1)}
    ||u_{\alpha gq}-b||_{L^{\infty}(L^2(\Gamma_1))}\leq K,
\end{equation}\end{small}
\newline for all $\alpha>1$, with $K$ depending of
$||\dot{u}_{gq}||_{L^{2}(V_0')}$, $||\dot{u}_{gq}||_{L^{2}(V')}$,
$||\nabla u_{gq}||_{\mathcal{H}}$, $||u_{gq}||_{L^{2}(V)}$,
$||u_{gq}||_{L^{\infty}(H)}$, $||g||_{\mathcal{H}}$,
$||q||_{\mathcal{Q}}$ and the coerciveness constant $\lambda_1$ of
the bilinear form $a_{1}$.
\end{proposition}
\begin{proof}
Taking $v=u_{\alpha gq}(t)-u_{gq}(t)\in V$ in the variational
equality (\ref{Palfavariacional}), taking into account that
$u_{gq}\big|_{\Gamma_1}=b$, and by using Young's inequality, we have
\begin{equation*}\begin{split}
&\quad\langle \dot{u}_{\alpha gq}(t)-\dot{u}_{gq}(t),u_{\alpha
gq}(t)-u_{gq}(t)\rangle+\frac{\lambda_1}{2}||u_{\alpha
gq}(t)-u_{gq}(t)||^2_V
\\&\quad+(\alpha-1)\int\limits_{\Gamma_1}(u_{\alpha gq}(t)-u_{gq}(t))^2d\gamma
\\& \leq
\frac{2}{\lambda_1}||g(t)||_{H}^2+\frac{2}{\lambda_1}\|\gamma_{0}\|^{2}||q(t)||^2_{Q}
+\frac{2}{\lambda_1}||\nabla u_{gq}(t)||^2_{H}+\frac{2}{\lambda_1}||
\dot{u}_{gq}(t)||^2_{V'}
\end{split}\end{equation*}
where $\gamma_{0}$ is the trace operator on $\Gamma$. Next, integrating between $0$ and $T$, and using that  $u_{\alpha gq}(0)=u_{gq}(0)=v_b$, we obtain
\begin{equation*}\begin{split}
& \frac{1}{2}||u_{\alpha
gq}(T)-u_{gq}(T)||_H^2+\frac{\lambda_1}{2}||u_{\alpha
gq}-u_{gq}||^2_{L^2(V)}+(\alpha-1)||u_{\alpha
gq}-b||^2_{L^2(L^2(\Gamma_1))}
\\& \leq
\frac{2}{\lambda_1}[||g||_{\mathcal{H}}^2+\|\gamma_{0}\|^{2}||q||^2_{\mathcal{Q}}
+||\nabla u_{gq}||^2_{\mathcal{H}}+|| \dot{u}_{gq}||^2_{L^2(V')}]\\
&=\frac{2}{\lambda_1}A,
\end{split}\end{equation*}
where $A=||g||_{\mathcal{H}}^2+\|\gamma_{0}\|^2||q||^2_{\mathcal{Q}}
+||\nabla u_{gq}||^2_{\mathcal{H}}+||\dot{u}_{gq}||^2_{L^2(V')}$.
From here, we deduce
\begin{equation}\label{71aaa}\sqrt{(\alpha-1)}||u_{\alpha gq}-b||_{L^2(L^2(\Gamma_1))}
\leq\sqrt{\frac{2}{\lambda_1}A}, \end{equation}
\begin{equation}\label{71}
||u_{\alpha gq}||_{L^2(V)}\leq
\frac{2}{\lambda_{1}}\sqrt{A}+||u_{gq}||_{L^2(V)}.
\end{equation}
\begin{equation}\label{30aaaa}\begin{split}
||u_{\alpha gq}||_{L^{\infty}(H)}\leq
\frac{2}{\sqrt{\lambda_1}}\sqrt{A}+||u_{gq}||_{L^{\infty}(H)}.\end{split}\end{equation}
Next, taking $v\in V_{0}$ in the variational equality
(\ref{Palfavariacional}) and subtracting with variational equality (\ref{Pvariacional}), we have
\begin{equation*}
( \dot{u}_{\alpha gq}(t)-\dot{u}_{gq}(t), v)_{H}+a(u_{\alpha
gq}(t)-u_{gq}(t),v)=0, \quad\forall v\in V_0.
\end{equation*}
Therefore
\begin{equation}\label{cuenta}
( \dot{u}_{\alpha gq}(t)-\dot{u}_{gq}(t), v)_{H} \leq
||u_{gq}(t)-u_{\alpha gq}(t)||_{V}||v||_{V_0}\quad\forall v\in V_0,
\end{equation}
taking supremum for $v\in V_0$ with $\|v\|_{V_{0}}\leq 1$ and integrating
between $0$ and $T$, we obtain $||\dot{u}_{\alpha gq}-\dot{u}_{gq}||_{L^2(V_0')} \leq
||u_{gq}-u_{\alpha gq}||_{L^2(V)}$ and therefore
\begin{equation}\label{71bbx}
||\dot{u}_{\alpha gq}||_{L^2(V_0')} \leq
\frac{2}{\lambda_{1}}\sqrt{A}+2||u_{gq}||_{L^2(V)}+||\dot{u}_{gq}||_{L^2(V_0')}.
\end{equation}
Finally, from (\ref{71aaa}), (\ref{71}), (\ref{30aaaa}) and (\ref{71bbx}), the thesis holds.
\end{proof}
\begin{proposition}\label{proposicion42bi} For fixed $(g,q)$, if $p_{\alpha gq}$ is the unique solution of the problem (\ref{palphavariacional}), then we have
\begin{equation}\label{estimacion2bi}
    ||\dot{p}_{\alpha gq}||_{L^{2}(V_0')}+||p_{\alpha gq}||_{L^{\infty}(H)}+
    ||p_{\alpha gq}||_{L^{2}(V)}
   +\sqrt{(\alpha-1)}
    ||p_{\alpha gq}||_{L^{2}(L^2(\Gamma_1))}\leq K,
\end{equation}
\newline
for all $\alpha>1$, where $K$ is depending of
$||\dot{p}_{gq}||_{L^{2}(V_0')}$, $||\dot{p}_{gq}||_{L^{2}(V')}$,
$||g||_{\mathcal{H}}$, $||q||_{\mathcal{Q}}$, $||\nabla
p_{gq}||_{\mathcal{H}}$, $||p_{gq}||_{L^{2}(V)}$,
$||p_{gq}||_{L^{\infty}(H)}$, $||z_d||_{\mathcal{H}},$
$||\dot{u}_{gq}||_{L^{2}(V')}$, $||\nabla u_{gq}||_{\mathcal{H}}$,
$||u_{gq}||_{L^{2}(V)}$, $||u_{gq}||_{L^{\infty}(H)}$ and the
coerciveness constant $\lambda_1$.
\end{proposition}
\begin{proof} From the variational equality (\ref{palphavariacional}) and with an analogous reasoning to
Proposition \ref{proposicion41bi}, we obtain the estimation (\ref{estimacion2bi}) as in \cite{MT}.
\end{proof}

\begin{theorem} \label{teorema41bi} For fixed $(g,q)\in \mathcal{H}\times\mathcal{Q}$, when $\alpha\rightarrow\infty$, we obtain:
\begin{itemize}
\item [i)] if $u_{gq}$ and $u_{\alpha gq}$ are the unique solutions to the variational problems (\ref{Pvariacional}) and (\ref{Palfavariacional}) respectively, then $u_{\alpha gq}\rightarrow u_{gq}$ strongly in  $L^2(V)\cap L^{\infty}(H)$
and $\dot{u}_{\alpha gq}\rightarrow \dot{u}_{gq}$  strongly in
$L^2(V_0')$.
\item [ii)] if $p_{gq}$ and $p_{\alpha gq}$ are the unique solutions to the variational problems (\ref{ecvarbi}) and (\ref{palphavariacional}) respectively, then
$p_{\alpha gq}\rightarrow p_{gq}$ strongly in  $L^2(V)\cap L^{\infty}(H)$
and $\dot{p}_{\alpha gq}\rightarrow \dot{p}_{gq}$ strongly in
$L^2(V_0')$.
\end{itemize}
\end{theorem}
\begin{proof}

(i) For fixed $(g,q)\in \mathcal{H}\times \mathcal{Q}$, we consider
a sequence $\{u_{\alpha_n gq}\} $ in $L^2(V)\cap L^{\infty}(H)$ and
by estimation (\ref{estimacion1bi}), we have that $||u_{\alpha_{n}
gq}||_{L^2(V)}\leq K$ and $||\dot{u}_{\alpha_n gq}||_{L^2(V_0')}$
$\leq K$, therefore, there exists a subsequence
$\{u_{\alpha_{n}gq}\}$ which is weakly convergent to $w_{gq}\in
L^2(V)$ and weakly* in $L^{\infty}(H)$ and there exists a
subsequence $\{\dot{u}_{\alpha_{n}gq}\}$ which is weakly convergent
to $\dot{w}_{gq}\in L^2(V_0')$. Now, from the weak lower
semicontinuity of the norm, we have that $w_{gq}=b$ on $\Gamma_1$
and therefore $w_{gq}-v_{b}\in L^2(V_{0})$.
\newline
Next, taking into account that $w_{gq}$ satisfies the following variational problem
\begin{equation*}
\left\{
\begin{array}{l l}
w_{gq}-v_b \in L^2(V_0), \qquad w_{gq}(0)=v_b\quad\text{and}\quad \dot{w}_{gq}\in L^2(V_0')\\
\text{such that}\quad \langle \dot{w}_{gq}(t),
v\rangle+a(w_{gq}(t),v)=L(t,v), \quad \forall v\in V_0,
\end{array}
\right.
\end{equation*}
and by uniqueness of the solution of the problem
(\ref{Pvariacional}), we have $w_{gq}=u_{gq}$.
\newline Therefore, when $\alpha_{n}\rightarrow\infty$ (called $\alpha\rightarrow\infty$), we get
$$u_{\alpha gq}\rightharpoonup u_{gq}\,\,\text{in}\,\,L^{2}(V),\,\,
u_{\alpha gq}\stackrel{*}{\rightharpoonup}
u_{gq}\,\,\text{in}\,\,L^{\infty}(H)\quad\,\,\text{and}\quad
\dot{u}_{\alpha gq}\rightharpoonup
\dot{u}_{gq}\,\,\text{in}\,\,L^{2}(V_0').$$
\newline
Now, we have
\small{\begin{equation*}\begin{split}
&\frac{1}{2}||u_{\alpha
gq}(T)-u_{gq}(T)||_H^2+\lambda_1||u_{\alpha
gq}-u_{gq}||^2_{L^2(V)}+(\alpha-1)||u_{\alpha
gq}-u_{gq}||^2_{L^2(L^2(\Gamma_1))}
\\&\leq \int\limits_0^T \{
L(t,u_{\alpha gq}(t)-u_{gq}(t))-a(u_{gq}(t),u_{\alpha
gq}(t)-u_{gq}(t))\\& \quad -\langle \dot{u}_{gq}(t),u_{\alpha
gq}(t)-u_{gq}(t)\rangle \} dt
\end{split}\end{equation*}}
and by using the weak convergence of $u_{\alpha gq}$ to $u_{gq}$, we
prove the strong convergence in $L^2(V)$ and the strong convergence
in $L^2(L^2(\Gamma_1))$. Now, from the variational equalities
(\ref{Pvariacional}) and (\ref{Palfavariacional}), we have as in
(\ref{cuenta}), that
\begin{equation*}
( \dot{u}_{\alpha gq}(t)-\dot{u}_{gq}(t), v)_{H} \leq
||u_{gq}(t)-u_{\alpha gq}(t)||_{V}||v||_{V_0},\quad\forall v\in V_0
\end{equation*}
then
\begin{equation*}
||\dot{u}_{\alpha gq}-\dot{u}_{gq}||^2_{L^2(V_0')} \leq
||u_{gq}-u_{\alpha gq}||^2_{L^2(V)}\rightarrow
0,\quad \text{when}\quad \alpha \rightarrow
\infty,
\end{equation*}
and we have that $\dot{u}_{\alpha gq}$ is strongly convergent to
$\dot{u}_{gq}$ in $L^2(V_0')$.

(ii) For fixed $(g,q)\in \mathcal{H} \times\mathcal{Q}$ we prove that there exists a sequence in $L^2(V)\cap L^{\infty}(H)$
 and $\eta_{gq}\in L^2(V)\cap
L^{\infty}(H)$ such that $p_{\alpha_n gq}\rightharpoonup \eta_{gq}$
weakly in $L^2(V)$ and weakly* in $L^{\infty}(H)$ and $\dot{p}_{\alpha_n gq}\rightharpoonup \dot{\eta}_{gq}$
weakly in $L^2(V_0')$. Next, we obtain that $\eta_{gq}$ verifies the variational problem (\ref{ecvarbi}), and by uniqueness of the solution we have that $\eta_{gq}=p_{gq}$.
Here, when $\alpha\rightarrow\infty$, we obtain that
$$p_{\alpha gq}\rightharpoonup p_{gq}\,\,\text{ in}\,\,L^{2}(V),
\quad p_{\alpha gq}\stackrel{*}{\rightharpoonup} p_{gq}\,\,\text{
in}\,\,L^{\infty}(H)\,\,\text{and}\,\,\, \dot{p}_{\alpha
gq}\rightharpoonup \dot{p}_{gq}\,\,\text{in}\,\,L^{2}(V_0').$$
Finally, the strong convergence of $p_{\alpha gq}$ to  $p_{gq}$ in
$L^2(V)\cap L^{\infty}(H)$ and of $\dot{p}_{\alpha gq}$ to
$\dot{p}_{gq}$ in $L^2(V_0')$ is obtained in a similar way that (i).
\end{proof}

Now, in the next theorem we prove the strong convergence of the
optimal controls, the system and the adjoint states of the optimal control problems (\ref{pcd2bi}) to the optimal control, the system and the adjoint states of the optimal control problem (\ref{pcd1bi}), when $\alpha\rightarrow
\infty$.

\begin{theorem}
a) If $u_{\overline{\overline{g}}\,\overline{\overline{q}}}$ and $u_{\alpha\overline{\overline{g}}_{\alpha}\overline{\overline{q}}_{\alpha}}$
 are the unique system states corresponding to the simultaneous optimal control problems (\ref{pcd1bi}) and (\ref{pcd2bi}) respectively, then we get
\begin{equation}\label{bi1}
(i) \lim\limits_{\alpha \rightarrow
\infty}||u_{\alpha\overline{\overline{g}}_{\alpha}\overline{\overline{q}}_{\alpha}}-u_{\overline{\overline{g}}\,\overline{\overline{q}}}||_{L^2(V)}=0,\,
(ii)\lim\limits_{\alpha \rightarrow
\infty}||\dot{u}_{\alpha\overline{\overline{g}}_{\alpha}\overline{\overline{q}}_{\alpha}}-\dot{u}_{\overline{\overline{g}}\,\overline{\overline{q}}}||_{L^2(V'_{0})}=0.
\end{equation}
b) If $p_{\overline{\overline{g}}\,\overline{\overline{q}}}$ and $p_{\alpha\overline{\overline{g}}_{\alpha}\overline{\overline{q}}_{\alpha}}$ are the unique adjoint states corresponding to the simultaneous optimal control problems (\ref{pcd1bi}) and (\ref{pcd2bi}) respectively, then
\begin{equation}\label{bi2}
(i) \lim\limits_{\alpha \rightarrow
\infty}||p_{\alpha\overline{\overline{g}}_{\alpha}\overline{\overline{q}}_{\alpha}}-p_{\overline{\overline{g}}\,\overline{\overline{q}}}||_{L^2(V)}=0,\,
(ii) \lim\limits_{\alpha \rightarrow
\infty}||\dot{p}_{\alpha\overline{\overline{g}}_{\alpha}\overline{\overline{q}}_{\alpha}}-\dot{p}_{\overline{\overline{g}}\,\overline{\overline{q}}}||_{L^2(V'_{0})}=0.
\end{equation}
c) If $(\overline{\overline{g}},\overline{\overline{q}})$ and $(\overline{\overline{g}}_{\alpha},\overline{\overline{q}}_{\alpha})$ are the unique solutions from the
simultaneous distributed-boundary optimal control problems (\ref{pcd1bi}) and (\ref{pcd2bi}) respectively, then
\begin{equation}\label{bi3}
\lim\limits_{\alpha \rightarrow
\infty}||(\overline{\overline{g}}_{\alpha},\overline{\overline{q}}_{\alpha})-(\overline{\overline{g}},\overline{\overline{q}})||_{\mathcal{H}\times\mathcal{Q}}=0.
\end{equation}
\end{theorem}

\begin{proof}
We will do the proof in three steps.

\textbf{Step 1.} From the estimation (\ref{estimacion1bi}) for
$u_{\alpha gq}$ with $g=q=0$, there exists a constant $K_{1}>0$ such
that $||u_{\alpha 00}||_{\mathcal{H}}\leq ||u_{\alpha
00}||_{L^2(V)}\leq K_{1}$, $\forall \alpha>1$. From the definition
of $J_{\alpha}$ and since
$J_{\alpha}(\overline{\overline{g}}_{\alpha},\overline{\overline{q}}_{\alpha})\leq
J_{\alpha}(0,0)$, we have:
\begin{equation}\frac{1}{2}||u_{\alpha\overline{\overline{g}}_{\alpha}\overline{\overline{q}}_{\alpha}}-z_d||^2_{\mathcal{H}}
+\frac{M_{1}}{2}||\overline{\overline{g}}_{\alpha}||^2_{\mathcal{H}}+\frac{M_{2}}{2}||\overline{\overline{q}}_{\alpha}||^2_{\mathcal{Q}}\leq
\frac{1}{2}||u_{\alpha
00}-z_d||^2_{\mathcal{H}}.\nonumber\end{equation} Therefore, there
exist positive constants $K_{2}$, $K_{3}$ and $K_{4}$ such that
$$||u_{\alpha\overline{\overline{g}}_{\alpha}\overline{\overline{q}}_{\alpha}}||_{\mathcal{H}}\leq K_2 , \quad ||\overline{\overline{g}}_{\alpha}||_{\mathcal{H}}
\leq K_{3} \quad
\text{and}\quad||\overline{\overline{q}}_{\alpha}||_{\mathcal{Q}}\leq
K_{4}.$$ Now, by estimation (\ref{estimacion1bi}) in Proposition
\ref{proposicion41bi}, we obtain that, for all $\alpha>1$ there
exists $K_{5}>0$ such that
\begin{equation}\label{d5}
||u_{\alpha\overline{\overline{g}}_{\alpha}\overline{\overline{q}}_{\alpha}}||_{L^{2}(V)}+
||\dot{u}_{\alpha\overline{\overline{g}}_{\alpha}\overline{\overline{q}}_{\alpha}}||_{L^2(V_0')}+\sqrt{(\alpha-1)}
    ||u_{\alpha\overline{\overline{g}}_{\alpha}\overline{\overline{q}}_{\alpha}}-b||_{L^{2}(L^2(\Gamma_1))}\leq K_5
\end{equation}
and by estimation (\ref{estimacion2bi}) in Proposition
\ref{proposicion42bi}, there exists a positive constant $K_{6}$ such
that
\begin{equation}\label{d6}
||p_{\alpha\overline{\overline{g}}_{\alpha}\overline{\overline{q}}_{\alpha}}||_{L^{2}(V)}+
||\dot{p}_{\alpha\overline{\overline{g}}_{\alpha}\overline{\overline{q}}_{\alpha}}||_{L^2(V_0')}+\sqrt{(\alpha-1)}
    ||p_{\alpha\overline{\overline{g}}_{\alpha}\overline{\overline{q}}_{\alpha}}||_{L^{2}(L^2(\Gamma_1))}\leq
    K_6.
\end{equation}
From the previous estimations, we have that there exist $f\in
\mathcal{H}$, $\delta\in \mathcal{Q}$, $\mu\in L^2(V)$,
$\dot{\mu}\in L^2(V_0')$, $\rho\in L^2(V)$ and $\dot{\rho}\in
L^2(V_0')$ such that
$$\overline{\overline{g}}_{\alpha}\rightharpoonup f\in \mathcal{H},\quad \overline{\overline{q}}_{\alpha}\rightharpoonup \delta\in \mathcal{Q}$$
$$u_{\alpha\overline{\overline{g}}_{\alpha}\overline{\overline{q}}_{\alpha}}\rightharpoonup \mu\in L^2(V),\quad
\dot{u}_{\alpha\overline{\overline{g}}_{\alpha}\overline{\overline{q}}_{\alpha}}\rightharpoonup
\dot{\mu}\in L^2(V'_{0}),$$
$$p_{\alpha\overline{\overline{g}}_{\alpha}\overline{\overline{q}}_{\alpha}}\rightharpoonup \rho\in L^2(V),\quad
\dot{p}_{\alpha\overline{\overline{g}}_{\alpha}\overline{\overline{q}}_{\alpha}}\rightharpoonup
\dot{\rho}\in L^2(V'_{0}).$$

\textbf{Step 2.}  Taking into account the weak convergence of
$u_{\alpha\overline{\overline{g}}_{\alpha}\overline{\overline{q}}_{\alpha}}$
to $\mu$ in $L^2(V)$ and the estimation (\ref{d5}) we obtain in
similar way to Theorem \ref{teorema41bi} (i), that $\mu=u_{f\delta}$. Moreover,
for the adjoint state, we have that
$p_{\alpha\overline{\overline{g}}_{\alpha}\overline{\overline{q}}_{\alpha}}$
is weakly convergent to $\rho$ in $L^2(V)$ and from estimation
(\ref{d6}) we obtain in similar way that Theorem \ref{teorema41bi} (ii), that
$\rho=p_{f\delta}$. Therefore, we have
\[
u_{\alpha\overline{\overline{g}}_{\alpha}\overline{\overline{q}}_{\alpha}}\rightharpoonup
u_{f\delta}\quad\text{in}\quad L^2(V)\quad \text{and}\quad
p_{\alpha\overline{\overline{g}}_{\alpha}\overline{\overline{q}}_{\alpha}}\rightharpoonup
p_{f\delta}\quad\text{in}\quad L^2(V).
\]
Now, the optimality condition for the problem (\ref{pcd2bi}) is
given by
\[
(M_{1}
\overline{\overline{g}}_{\alpha}+p_{\alpha\overline{\overline{g}}_{\alpha}\overline{\overline{q}}_{\alpha}},h)_{\mathcal{H}}+
(M_{2}\overline{\overline{q}}_{\alpha}-p_{\alpha\overline{\overline{g}}_{\alpha}\overline{\overline{q}}_{\alpha}},\eta)_{\mathcal{Q}}=0
\quad \forall (h,\eta) \in \mathcal{H}\times \mathcal{Q}
\]
and taking into account that
\[
p_{\alpha\overline{\overline{g}}_{\alpha}\overline{\overline{q}}_{\alpha}}\rightharpoonup
p_{f\delta}\text{ in } L^2(V),\quad
\overline{\overline{g}}_{\alpha}\rightharpoonup f\in
\mathcal{H},\quad \overline{\overline{q}}_{\alpha}\rightharpoonup
\delta\in \mathcal{Q}
\]
we obtain
\[
(M_{1} f+p_{f\delta},h)_{\mathcal{H}}+
(M_{2}\delta-p_{f\delta},\eta)_{\mathcal{Q}}=0 \quad \forall
(h,\eta) \in \mathcal{H}\times \mathcal{Q}
\]
and by uniqueness of the optimal control we deduce that
$f=\overline{\overline{g}}$ and $\delta=\overline{\overline{q}}$.
Therefore
$u_{f\delta}=u_{\overline{\overline{g}}\,\overline{\overline{q}}}$,
$p_{f\delta}=p_{\overline{\overline{g}}\,\overline{\overline{q}}}$,
$\dot{u}_{f\delta}=\dot{u}_{\overline{\overline{g}}\,\overline{\overline{q}}}$
and
$\dot{p}_{f\delta}=\dot{u}_{\overline{\overline{g}}\,\overline{\overline{q}}}$.

\textbf{Step 3.} We have, for all $(g,q) \in \mathcal{H}\times
\mathcal{Q}$
\begin{equation*}\begin{split}
J(\overline{\overline{g}},\overline{\overline{q}})&=\frac{1}{2}||u_{\overline{\overline{g}}\,\overline{\overline{q}}}-z_d||_{\mathcal{H}}^2
+\frac{M_1}{2}||\overline{\overline{g}}||^2_{\mathcal{H}}+\frac{M_2}{2}||\overline{\overline{q}}||^2_{\mathcal{Q}}
\\&\leq\liminf\limits_{\alpha\rightarrow\infty}\left[\frac{1}{2}||u_{\alpha\overline{\overline{g}}_{\alpha}\overline{\overline{q}}_{\alpha}}-z_d||_{\mathcal{H}}^2
+\frac{M_1}{2}||\overline{\overline{g}}_{\alpha}||^2_{\mathcal{H}}+\frac{M_2}{2}||\overline{\overline{q}}_{\alpha}||^2_{\mathcal{Q}}\right]
\\&\leq\limsup\limits_{\alpha\rightarrow\infty}\left[\frac{1}{2}||u_{\alpha\overline{\overline{g}}_{\alpha}\overline{\overline{q}}_{\alpha}}-z_d||_{\mathcal{H}}^2+
\frac{M_1}{2}||\overline{\overline{g}}_{\alpha}||^2_{\mathcal{H}}+\frac{M_2}{2}||\overline{\overline{q}}_{\alpha}||^2_{\mathcal{Q}}\right]\\&
\leq\limsup\limits_{\alpha\rightarrow\infty}
J_{\alpha}(g,q)\\&=\lim\limits_{\alpha\rightarrow\infty}\left[\frac{1}{2}||u_{\alpha
gq}-z_d||_{\mathcal{H}}^2+\frac{M_1}{2}||g||^2_{\mathcal{H}}+\frac{M_2}{2}||q||^2_{\mathcal{Q}}\right]\\&=\frac{1}{2}||u_{gq}-z_d||_{\mathcal{H}}^2
+\frac{M_1}{2}||g||^2_{\mathcal{H}}+\frac{M_2}{2}||q||^2_{\mathcal{Q}}=J(g,q).
\end{split}\end{equation*}
By taking infimum on $(g,q)$, all the above inequalities become
equalities and therefore we get
\begin{equation*}\begin{split}
&\,\quad\lim\limits_{\alpha\rightarrow\infty}\left[\frac{1}{2}||u_{\alpha\overline{\overline{g}}_{\alpha}\overline{\overline{q}}_{\alpha}}-z_d||_{\mathcal{H}}^2
+\frac{M_1}{2}||\overline{\overline{g}}_{\alpha}||^2_{\mathcal{H}}+\frac{M_2}{2}||\overline{\overline{q}}_{\alpha}||^2_{\mathcal{Q}}\right]
\\&=\frac{1}{2}||u_{\overline{\overline{g}}\,\overline{\overline{q}}}-z_d||_{\mathcal{H}}^2+\frac{M_1}{2}||\overline{\overline{g}}||^2_{\mathcal{H}}+\frac{M_2}{2}||
\overline{\overline{q}}||^2_{\mathcal{Q}},
\end{split}\end{equation*}
that is
\begin{equation*}\begin{split}
&\quad\lim\limits_{\alpha\rightarrow\infty}||(\sqrt{M_2}
\overline{\overline{q}}_{\alpha},\sqrt{M_1}\overline{\overline{g}}_{\alpha},u_{\alpha\overline{\overline{g}}_{\alpha}\overline{\overline{q}}_{\alpha}}-z_d)||_{\mathcal{Q}\times
\mathcal{H}\times \mathcal{H}}^2\\&=||(\sqrt{M_2}
\overline{\overline{q}},\sqrt{M_1}\overline{\overline{g}},u_{\overline{\overline{g}}\,\overline{\overline{q}}}-z_d)||^2_{\mathcal{Q}\times
\mathcal{H}\times \mathcal{H}}.
\end{split}\end{equation*}
The previous equality, the convergence
$\overline{\overline{q}}_{\alpha}\rightharpoonup
\overline{\overline{q}}$ in $\mathcal{Q}$,
$\overline{\overline{g}}_{\alpha}\rightharpoonup
\overline{\overline{g}}$ in $\mathcal{H}$ and
$u_{\alpha\overline{\overline{g}}_{\alpha}\overline{\overline{q}}_{\alpha}}\rightharpoonup
u_{\overline{\overline{g}}\,\overline{\overline{q}}}$ in $L^2(V)$,
imply that
$(\overline{\overline{q}}_{\alpha},\overline{\overline{g}}_{\alpha},u_{\alpha\overline{\overline{g}}_{\alpha}\overline{\overline{q}}_{\alpha}})
\rightarrow(\overline{\overline{q}},\overline{\overline{g}},u_{\overline{\overline{g}}\,\overline{\overline{q}}})$,
when $\alpha\rightarrow\infty$, then (\ref{bi3}) holds. Finally, if
we take
$v=u_{\alpha\overline{\overline{g}}_{\alpha}\overline{\overline{q}}_{\alpha}}(t)-u_{\overline{\overline{g}}\,\overline{\overline{q}}}(t)\in
V$ in (\ref{Palfavariacional}) for
$u_{\alpha\overline{\overline{g}}_{\alpha}\overline{\overline{q}}_{\alpha}}$,
we have
\begin{equation*}\begin{split}
&\quad\langle
\dot{u}_{\alpha\overline{\overline{g}}_{\alpha}\overline{\overline{q}}_{\alpha}}(t)-\dot{u}_{\overline{\overline{g}}\,\overline{\overline{q}}}(t),
u_{\alpha\overline{\overline{g}}_{\alpha}\overline{\overline{q}}_{\alpha}}(t)-u_{\overline{\overline{g}}\,\overline{\overline{q}}}(t)\rangle
+\lambda_1||u_{\alpha\overline{\overline{g}}_{\alpha}\overline{\overline{q}}_{\alpha}}(t)-u_{\overline{\overline{g}}\,\overline{\overline{q}}}(t)||_{V}^2
\\&\quad+(\alpha-1)\int\limits_{\Gamma_1}(u_{\alpha\overline{\overline{g}}_{\alpha}\overline{\overline{q}}_{\alpha}}(t)-u_{\overline{\overline{g}}\,
\overline{\overline{q}}}(t))^2d\gamma
\\&\leq (\overline{\overline{g}}_{\alpha}(t)-\dot{u}_{\overline{\overline{g}}\,\overline{\overline{q}}}(t),u_{\alpha\overline{\overline{g}}_{\alpha}
\overline{\overline{q}}_{\alpha}}(t)-u_{\overline{\overline{g}}\,\overline{\overline{q}}}(t))_H-(\overline{\overline{q}}_{\alpha}(t),
u_{\alpha\overline{\overline{g}}_{\alpha}\overline{\overline{q}}_{\alpha}}(t)-u_{\overline{\overline{g}}\,\overline{\overline{q}}}(t))_Q
\\&\quad-a(u_{\overline{\overline{g}}\,\overline{\overline{q}}}(t),u_{\alpha\overline{\overline{g}}_{\alpha}\overline{\overline{q}}_{\alpha}}(t)
-u_{\overline{\overline{g}}\,\overline{\overline{q}}}(t)).
\end{split}\end{equation*}
If we call
$z_{\alpha}=u_{\alpha\overline{\overline{g}}_{\alpha}\overline{\overline{q}}_{\alpha}}-u_{\overline{\overline{g}}\,\overline{\overline{q}}}$, from the previous equality, we obtain
\begin{equation*} \lambda_1||z_{\alpha}(t)||_{V}^2
\leq
(\overline{\overline{g}}_{\alpha}(t)-\dot{u}_{\overline{\overline{g}}\,\overline{\overline{q}}}(t),z_{\alpha}(t))_H-(\overline{\overline{q}}_{\alpha}(t),z_{\alpha}(t))_Q-a(u_{\overline{\overline{g}}\,\overline{\overline{q}}}(t),z_{\alpha}(t)),
\end{equation*}
and integrating between 0 and $T$, we have
\small{\begin{equation*}\begin{split}
\lambda_1||z_{\alpha}||_{L^2(V)}^2 &
=\lambda_1\int\limits_0^T||z_{\alpha}(t)||_{V}^2dt
\\&\leq\int\limits_0^T\left[ (\overline{\overline{g}}_{\alpha}(t)-\dot{u}_{\overline{\overline{g}}\,
\overline{\overline{q}}}(t),z_{\alpha}(t))_H-(\overline{\overline{q}}_{\alpha}(t),z_{\alpha}(t))_Q-a(u_{\overline{\overline{g}}\,
\overline{\overline{q}}}(t),z_{\alpha}(t))\right]dt.
\end{split}\end{equation*}}
Since $z_{\alpha}\rightharpoonup 0$ weakly in $L^2(V)$,
$\overline{\overline{q}}_{\alpha}\rightarrow
\overline{\overline{q}}$ strongly in $\mathcal{Q}$ and
$\overline{\overline{g}}_{\alpha}\rightarrow
\overline{\overline{g}}$ strongly in $\mathcal{H}$ we obtain, when
$\alpha\rightarrow\infty$
$$\int\limits_0^T\left[
(\overline{\overline{g}}_{\alpha}(t)-\dot{u}_{\overline{\overline{g}}\,\overline{\overline{q}}}(t),z_{\alpha}(t))_H-(\overline{\overline{q}}_{\alpha}(t),z_{\alpha}(t))_Q-a(u_{\overline{\overline{g}}\,\overline{\overline{q}}}(t),z_{\alpha}(t))\right]dt\rightarrow
0,$$ then (\ref{bi1} i) holds.
\newline From the variational equalities (\ref{Pvariacional}) and
(\ref{Palfavariacional}), we have
\begin{equation*}\langle \dot{z}_{\alpha}(t),
v\rangle+a(z_{\alpha}(t),v)=(\overline{\overline{g}}_{\alpha}(t)-\overline{\overline{g}}(t),v)_H+
(\overline{\overline{q}}(t)-\overline{\overline{q}}_{\alpha}(t),v)_Q,
\quad \forall v\in V_0,
\end{equation*}
and therefore there exists a positive constant $K_{7}$ such that
\begin{equation*}|| \dot{z}_{\alpha}||^2_{L^2(V_0')}\leq K_{7}\left[|| z_{\alpha}||^2_{L^2(V)}+||\overline{\overline{g}}_{\alpha}-\overline{\overline{g}}||
^2_{\mathcal{H}}+||\overline{\overline{q}}-\overline{\overline{q}}_{\alpha}||^2_{\mathcal{Q}}\right].
\end{equation*}
Since  $\overline{\overline{q}}_{\alpha}\rightarrow
\overline{\overline{q}}$ strongly in $\mathcal{Q}$,
$\overline{\overline{g}}_{\alpha}\rightarrow
\overline{\overline{g}}$ strongly in $\mathcal{H}$ and
$u_{\alpha\overline{\overline{g}}_{\alpha}\overline{\overline{q}}_{\alpha}}\rightarrow
u_{\overline{\overline{g}}\,\overline{\overline{q}}}$ strongly in
$L^2(V)$ when $\alpha\rightarrow\infty$, we can say that
$\dot{z}_{\alpha}\rightarrow 0$ strongly in $L^2(V_0')$, that is
$\dot{u}_{\alpha\overline{\overline{g}}_{\alpha}\overline{\overline{q}}_{\alpha}}\rightarrow
\dot{u}_{\overline{\overline{g}}\,\overline{\overline{q}}}$ strongly
in $L^2(V_0')$, then (\ref{bi1} ii) holds.
\newline In similar way, we prove that $(p_{\alpha\overline{\overline{g}}_{\alpha}\overline{\overline{q}}_{\alpha}},
\dot{p}_{\alpha\overline{\overline{g}}_{\alpha}\overline{\overline{q}}_{\alpha}})\rightarrow
(p_{\overline{\overline{g}}\,
\overline{\overline{q}}},\dot{p}_{\overline{\overline{g}}\,\overline{\overline{q}}})$
strongly in $L^2(V)\times L^2(V_0')$, when $\alpha\rightarrow
\infty$.
\end{proof}

\section{Estimations  between the optimal controls}

In this Section, we study the relation between the solutions of the
distributed optimal control problems given in \cite{MT} with the
solutions of the simultaneous distributed boundary optimal control
problems (\ref{pcd1bi}) and (\ref{pcd2bi}). Moreover, we give a
characterization of the solutions of these problems by using the
fixed point theory.

\subsection{Estimations with respect to the problem $P$}

We consider the distributed optimal control problem
\begin{equation}\label{41}
\text{find}\quad \overline{g}\in\mathcal{H} \quad\text{such that}\quad J_1(\overline{g})=\min\limits_{g\in\mathcal{H}}J_1(g)\qquad \text{for fixed}\quad q\in\mathcal{Q},
\end{equation}
where $J_1$ is the  cost functional defined in \cite{MT} plus the constant $\frac{M_2}{2}||q||^2_{\mathcal{Q}}$, that is, $J_1:\mathcal{H}\rightarrow \mathbb{R}^+_0$  is given by
\begin{equation*}
J_1(g)=\frac{1}{2}||u_g-z_d||^2_{\mathcal{H}}+\frac{M_1}{2}||g||^2_{\mathcal{H}} +\frac{M_2}{2}||q||^2_{\mathcal{Q}}    \qquad ( \text{fixed} \,\,q\in \mathcal{Q}),
\end{equation*}
where $u_g$ is the unique solution of the problem (\ref{Pvariacional}) for fixed $q$.
\begin{remark} The functional $J$ defined in (\ref{pcd1bi}) and the functional $J_1$ previously given, satisfy the following elemental estimation
\begin{equation*}
J(\overline{\overline{g}},\overline{\overline{q}})\leq J_1(\overline{g}),\quad \forall q\in\mathcal{Q}.
\end{equation*}\end{remark}
In the following theorem we obtain estimations between the solution of the distributed optimal control problem (\ref{41}) and the first  component of the solution of the simultaneous distributed-boundary optimal control problem (\ref{pcd1bi}).
\begin{theorem}
If
$(\overline{\overline{g}},\overline{\overline{q}})\in\mathcal{H}\times\mathcal{Q}$
is the unique solution of the distributed-boundary optimal control
problem (\ref{pcd1bi}),  $\overline{g}$ is the unique solution of
the optimal control problem (\ref{41}), then
\begin{equation}\label{450}
||\overline{g}-\overline{\overline{g}}||_{\mathcal{H}}\leq\frac{1}{\lambda_0 M_1}||u_{\overline{\overline{g}}\,\overline{\overline{q}}}-u_{\overline{g}\,\overline{\overline{q}}}||_{\mathcal{H}}.
\end{equation}
\end{theorem}
\begin{proof}
From the optimality condition for $\overline{g}$, with $q=\overline{\overline{q}}$, we have $(M_1\overline{g}+p_{\overline{g}\,\overline{\overline{q}}},h)_{\mathcal{H}}=0, \forall h\in\mathcal{H}$,
and taking $h=\overline{\overline{g}}-\overline{g}$, we obtain \begin{equation}\label{47}
(M_1\overline{g}+p_{\overline{g}\,\overline{\overline{q}}},\overline{\overline{g}}-\overline{g})_{\mathcal{H}}=0.
\end{equation}
On the other hand, if we take $\eta=0\in\mathcal{Q}$ in the optimality condition for $(\overline{\overline{g}},\overline{\overline{q}})$ we have $(M_1\overline{\overline{g}}+p_{\overline{\overline{g}}\,\overline{\overline{q}}},h)_{\mathcal{H}}=0, \forall h\in\mathcal{H}$, next, taking  $h=\overline{g}-\overline{\overline{g}}$, we obtain
\begin{equation}\label{48}
(M_1\overline{\overline{g}}+p_{\overline{\overline{g}}\,\overline{\overline{q}}},\overline{g}-\overline{\overline{g}})_{\mathcal{H}}=0.
\end{equation}
By adding (\ref{47}) and (\ref{48}), we have
$\left(M_1(\overline{g}-\overline{\overline{g}})+p_{\overline{g}\,\overline{\overline{q}}}-p_{\overline{\overline{g}}\,\overline{\overline{q}}},\overline{\overline{g}}-\overline{g}\right)_{\mathcal{H}}=0$. Here, we deduce that
\begin{equation*}
||\overline{\overline{g}}-\overline{g}||_{\mathcal{H}} \leq \frac{1}{M_1}||p_{\overline{g}\,\overline{\overline{q}}}-p_{\overline{\overline{g}}\,\overline{\overline{q}}}||_{L^2(V)}.
\end{equation*}
Next, by using the variational equality (\ref{ecvarbi}) for
$g=\overline{\overline{g}}$ and $q=\overline{\overline{q}}$, and for
$g=\overline{g}$ and $q=\overline{\overline{q}}$, respectively, we
obtain
\begin{equation*}
-\frac{d}{dt}
||p_{\overline{\overline{g}}\,\overline{\overline{q}}}(t)-p_{\overline{g}\,\overline{\overline{q}}}(t)||^2_H+\lambda_0
 ||p_{\overline{\overline{g}}\,\overline{\overline{q}}}(t)-p_{\overline{g}\,\overline{\overline{q}}}(t)||^2_V\leq \frac{1}{\lambda_0}||u_{\overline{\overline{g}}\,\overline{\overline{q}}}(t)-u_{\overline{g}\,\overline{\overline{q}}}(t)||^2_H.
\end{equation*}
By integrating between $0$ and $T$, and using that
$p_{\overline{\overline{g}}\,\overline{\overline{q}}}(T)=p_{\overline{g}\,\overline{\overline{q}}}(T)=0,$ we deduce
\begin{equation*}\begin{split}
&\quad
||p_{\overline{\overline{g}}\,\overline{\overline{q}}}(0)-p_{\overline{g}\,\overline{\overline{q}}}(0)||^2_H+\lambda_0
 ||p_{\overline{\overline{g}}\,\overline{\overline{q}}}-p_{\overline{g}\,\overline{\overline{q}}}||^2_{L^2(V)}\leq \frac{1}{\lambda_0}||u_{\overline{\overline{g}}\,\overline{\overline{q}}}-u_{\overline{g}\,\overline{\overline{q}}}||^2_{\mathcal{H}},
\end{split}\end{equation*}
then
\begin{equation*}\begin{split}
&\quad
 ||p_{\overline{\overline{g}}\,\overline{\overline{q}}}-p_{\overline{g}\,\overline{\overline{q}}}||_{L^2(V)}\leq \frac{1}{\lambda_0}||u_{\overline{\overline{g}}
 \,\overline{\overline{q}}}-u_{\overline{g}\,\overline{\overline{q}}}||_{\mathcal{H}},
\end{split}\end{equation*}
and therefore (\ref{450}) holds.
\end{proof}
Now, we will give a characterization of the solution of the simultaneous optimal control problem (\ref{pcd1bi}) by using the fixed point theory. For this, we introduce the operator $W:\mathcal{H}\times\mathcal{Q}\rightarrow\mathcal{H}\times\mathcal{Q}$, defined by
$$W(g,q)=\left(-\frac{1}{M_1}p_{gq},\frac{1}{M_2}p_{gq}\right).$$
\begin{theorem}\label{teoremita} There exists a positive constant  $C_0=C_0(\lambda_0,\gamma_0,M_1,M_2)$ such that, $\forall(g_1,q_1),(g_2,q_2)\in\mathcal{H}\times\mathcal{Q}$
\begin{equation*}
||W(g_2,q_2)-W(g_1,q_1)||_{\mathcal{H}\times\mathcal{Q}}\leq C_0||(g_2,q_2)-(g_1,q_1)||_{\mathcal{H}\times\mathcal{Q}},
\end{equation*}
and $W$ is a contraction if and only if the data satisfies the following condition
\begin{equation}\label{4280}
C_0=\frac{2}{\lambda_0^2}\sqrt{\frac{1}{M_1^2}+\frac{||\gamma_0||^2}{M_2^2}}\left(1+||\gamma_0||\right)<1.
\end{equation}
\end{theorem}
\begin{proof}
First, we prove the following estimates, $\forall(g_1,q_1),\,(g_2,q_2)\in \mathcal{H}\times\mathcal{Q}$
\begin{equation}\label{429}
||u_{g_1q_1}-u_{g_2q_2}||_{L^2(V)}\leq\frac{\sqrt{2}}{\lambda_0}\left(||g_2-g_1||_{\mathcal{H}}+||\gamma_0||\,||q_2-q_1||_{\mathcal{Q}}\right),
\end{equation}
\begin{equation}\label{430}
||p_{g_1q_1}-p_{g_2q_2}||_{L^2(V)}\leq\frac{1}{\lambda_0}||u_{g_1q_1}-u_{g_2q_2}||_{\mathcal{H}}.
\end{equation}
In fact, for the estimation (\ref{429}), we consider the variational equation (\ref{Pvariacional}) for $g=g_1$ and $q=q_1$, and for $g=g_2$ and $q=q_2$, respectively. Next, we obtain
\begin{equation*}\begin{split}
&\quad\frac{1}{2}\frac{d}{dt}||u_{g_1q_1}(t)-u_{g_2q_2}(t)||^2_H+\lambda_0 ||u_{g_1q_1}(t)-u_{g_2q_2}(t)||^2_{V}\\&\leq||u_{g_1q_1}(t)-u_{g_2q_2}(t)||_V\left(||g_1(t)-g_2(t)||_H+||\gamma_0||\,||q_1(t)-q_2(t)||_Q\right).
\end{split}\end{equation*}
Now, by Young's inequality and integrating between $0$ and $T$, we
deduce (\ref{429}). If we consider the variational equality
(\ref{ecvarbi}) for $g=g_1$ and $q=q_1$, and for $g=g_2$ and
$q=q_2$, respectively, we obtain (\ref{430}).
\newline
Finally, using the estimations (\ref{429}) and (\ref{430}), we
obtain \vspace{.2cm}
\newline
 $||W(g_2,q_2)-W(g_1,q_1)||_{\mathcal{H}\times\mathcal{Q}}\leq \frac{2}{\lambda_0^2}\sqrt{\frac{1}{M_1^2}+\frac{||\gamma_0||^2}{M_2^2}}
 \left(1+||\gamma_0||\right)||(g_2,q_2)-(g_1,q_1)||_{\mathcal{H}\times\mathcal{Q}},$
\vspace{.02cm}\newline and the operator $W$ is a contraction if and
only if (\ref{4280}) holds.
\end{proof}
\begin{corollary} If the data satisfy that $C_{0}<1$, then the unique solution $(\overline{\overline{g}},\overline{\overline{q}})\in\mathcal{H}\times\mathcal{Q}$ of the optimal control problem (\ref{pcd1bi}) can be characterized as the unique fixed point of the operator $W$, that is
\begin{equation*}
W(\overline{\overline{g}},\overline{\overline{q}})=\left(-\frac{1}{M_1}p_{\overline{\overline{g}}\,\overline{\overline{q}}},\frac{1}{M_2}p_{\overline{\overline{g}}\,\overline{\overline{q}}}\right)=(\overline{\overline{g}},\overline{\overline{q}}).
\end{equation*}
\end{corollary}
\begin{proof} When $C_0<1$, the operator $W$ is a contraction defined on $\mathcal{H}\times\mathcal{Q}$. Next, there exists a unique $(g^*,q^*)\in \mathcal{H}\times\mathcal{Q}$ such that
\begin{equation*}
W(g^*,q^*)=\left(-\frac{1}{M_1}p_{g^*q^*},\frac{1}{M_2}p_{g^*q^*}\right)=(g^*,q^*),
\end{equation*}
o equivalently
\begin{equation*}
\left(M_1g^*+p_{g^*q^*},M_2q^*-p_{g^*q^*}\right)=(0,0).
\end{equation*}
Here, $(g^*,q^*)$ verifies the optimality condition for the problem
(\ref{pcd1bi}), therefore the unique fixed point of $W$ is the
solution
$(\overline{\overline{g}},\overline{\overline{q}})\in\mathcal{H}\times\mathcal{Q}$
of this simultaneous optimal control problem.
\end{proof}

\subsection{Estimations with respect to the problem $P_{\alpha}$}

For each $\alpha>0$, we consider the following optimal control problem
\begin{equation}\label{41a}
\text{find}\quad \overline{g}_{\alpha}\in\mathcal{H} \quad
\text{such that}\quad
J_{1\alpha}(\overline{g}_{\alpha})=\min\limits_{g\in\mathcal{H}}J_{1\alpha}(g),
\quad \text{ for fixed}\quad q\in \mathcal{Q}
\end{equation}
where $J_{1\alpha}:\mathcal{H}\rightarrow \mathbb{R}^+_0$ is given by
\begin{equation*}
J_{1\alpha}(g)=\frac{1}{2}||u_{\alpha g}-z_d||^2_{\mathcal{H}}+\frac{M_1}{2}||g||^2_{\mathcal{H}} +\frac{M_2}{2}||q||^2_{\mathcal{Q}}    \qquad (\text{fixed}\quad q\in \mathcal{Q}),
\end{equation*}
that is, $J_{1\alpha}$ is the cost functional given in \cite{MT} plus the constant $\frac{M_2}{2}||q||^2_{\mathcal{Q}}$ and
 $u_{\alpha g}$ is the unique solution of the problem (\ref{Palfavariacional}) for fixed $q$.
\begin{remark} For each $\alpha >0$, the functional $J_{\alpha}$ defined in (\ref{pcd2bi}) and the functional $J_{1\alpha}$ previously given satisfy the following estimate
\begin{equation*}
J_{\alpha}(\overline{\overline{g}}_{\alpha},\overline{\overline{q}}_{\alpha})\leq J_{1\alpha}(\overline{g}_{\alpha}),\quad \forall q\in\mathcal{Q}.
\end{equation*}\end{remark}
An estimation between the solution of the distributed optimal
control problem (\ref{41a}) with the first component of the solution
of the simultaneous distributed-boundary optimal control problem
(\ref{pcd2bi}) is given in the following theorem whose prove is
omitted.
\begin{theorem}
If
$(\overline{\overline{g}}_{\alpha},\overline{\overline{q}}_{\alpha})\in\mathcal{H}\times\mathcal{Q}$
is the unique solution of the simultaneous optimal control problem
(\ref{pcd2bi}),  $\overline{g}_{\alpha}$ is the unique solution of
the optimal control problem (\ref{41a}), then
\begin{equation*}\label{45}
||\overline{g}_{\alpha}-\overline{\overline{g}}_{\alpha}||_{\mathcal{H}}\leq\frac{1}{\lambda_1\min\{1,\alpha \} M_1}||u_{\alpha\overline{\overline{g}}_{\alpha}\overline{\overline{q}}_{\alpha}}-u_{\alpha\overline{g}_{\alpha}\overline{\overline{q}}_{\alpha}}||_{\mathcal{H}}.
\end{equation*}
\end{theorem}
In similar way to Theorem \ref{teoremita}, we give a
characterization of the solution of the problem (\ref{pcd2bi})
proving that the operator $W_{\alpha}$, which is defined after, is a
contraction. This result is presented in the following theorem,
whose prove is omitted.

Let the operator
$W_{\alpha}:\mathcal{H}\times\mathcal{Q}\rightarrow\mathcal{H}\times\mathcal{Q}$,
for each $\alpha>0,$ defined by the expression
$$W_{\alpha}(g,q)=\left(-\frac{1}{M_1}p_{\alpha gq},\frac{1}{M_2}p_{\alpha gq}\right).$$
\begin{theorem}$W_{\alpha}$ is a Lipschitz operator on $\mathcal{H}\times \mathcal{Q}$, that is, there exists a positive constant $C_{0\alpha}=C_{0\alpha}(\lambda_1\min\{1,\alpha \},\gamma_0,M_1,M_2)$, such that $\forall(g_1,q_1)$, $(g_2,q_2)\in\mathcal{H}\times\mathcal{Q}$
\begin{equation*}
||W_{\alpha}(g_2,q_2)-W_{\alpha}(g_1,q_1)||_{\mathcal{H}\times\mathcal{Q}}\leq C_{0\alpha}||(g_2,q_2)-(g_1,q_1)||_{\mathcal{H}\times\mathcal{Q}}
\end{equation*}
and $W_{\alpha}$ is a contraction if and only if the data satisfy
the following inequality
\begin{equation*}\label{428}
C_{0\alpha}=\frac{2}{\lambda_{1}^2(\min\{1,\alpha \})^2}\sqrt{\frac{1}{M_1^2}+\frac{||\gamma_0||^2}{M_2^2}}\left(1+||\gamma_0||\right)<1.
\end{equation*}
\end{theorem}
\begin{corollary}If the data satisfy the condition $C_{0\alpha}<1$, then the unique solution
$(\overline{\overline{g}}_{\alpha},\overline{\overline{q}}_{\alpha})\in\mathcal{H}\times\mathcal{Q}$
of the problem (\ref{pcd2bi}) can be obtained as the unique fixed
point of the operator $W_{\alpha}$, that is
\begin{equation*}
W_{\alpha}(\overline{\overline{g}}_{\alpha},\overline{\overline{q}}_{\alpha})=
\left(-\frac{1}{M_1}p_{\alpha\overline{\overline{g}}_{\alpha}\overline{\overline{q}}_{\alpha}},\frac{1}{M_2}
p_{\alpha\overline{\overline{g}}_{\alpha}\overline{\overline{q}}_{\alpha}}\right)=(\overline{\overline{g}}_{\alpha},\overline{\overline{q}}_{\alpha}).
\end{equation*}
\end{corollary}

\section*{Acknowledgments} The present work has been partially sponsored
by the European Union's Horizon 2020 Research and Innovation
Programme under the Marie Sklodowska-Curie grant agreement 823731
CONMECH and by the Project PIP No. 0275 from CONICET - UA, Rosario,
Argentina for the first and third authors; by the Project ANPCyT
PICTO Austral 2016 No. 0090 for the first author, and by the Project
PPI No. 18/C468 from SECyT-UNRC, R\'io Cuarto, Argentina for the
second and third authors.

% You may incorporate your references as follows in your main tex file.
% Using BibTex is not recommended but can be handled.

\end{document}